\documentclass[reqno]{amsart}
\usepackage{amsmath}

\usepackage{amsfonts}
\usepackage{amssymb}
\usepackage{hyperref}

\begin{document}
\title[Existence results for Caputo type fractional differential equations]{Existence results for a class of Caputo type fractional differential equations with Riemann-Liouville fractional integrals and Caputo fractional derivatives in boundary conditions}
\author[F. Haddouchi]{Faouzi Haddouchi}
\address{
Department of Physics, University of Sciences and Technology of
Oran-MB, El Mnaouar, BP 1505, 31000 Oran, Algeria
\newline
And
\newline
Laboratory of Fundamental and Applied Mathematics of Oran (LMFAO)\\
University of Oran 1 Ahmed Benbella, 31000 Oran,
Algeria}
\email{fhaddouchi@gmail.com}
\begin{abstract}
In this paper, we investigate the existence and uniqueness of solutions for a fractional
boundary value problem supplemented with nonlocal Riemann-Liouville fractional integral and Caputo fractional derivative boundary conditions. Our results are based on some known tools of fixed point theory. Finally, some illustrative examples are included to verify the validity of our results.
\end{abstract}

\maketitle \numberwithin{equation}{section}
\newtheorem{theorem}{Theorem}[section]
\newtheorem{lemma}[theorem]{Lemma}
\newtheorem{definition}[theorem]{Definition}
\newtheorem{proposition}[theorem]{Proposition}
\newtheorem{corollary}[theorem]{Corollary}
\newtheorem{remark}[theorem]{Remark}
\newtheorem{exmp}{Example}[section]

\section{Introduction}

Fractional calculus is a generalization of ordinary differentiation and integration
to arbitrary order. Fractional differential equations have recently proved to
be valuable tools in many fields, such as viscoelasticity, engineering, physics, chemistry, mechanics, and economics, see \cite{Meral}, \cite{Oldham}, \cite{Nigmatullin}, \cite{Orsingher}, \cite{Kilbas}.
In the recent years, there has been a significant development in ordinary and partial differential equations involving fractional derivatives, see the monographs of Kilbas et al. \cite{Kilbas}, Miller and Ross \cite{Miller},
\cite{Podlubny}, and the papers \cite{Bakakhani, Lakshmikantham, Zhou1, Zhou2, Kosmatov, Kaufmann, Bai} and the references therein.

Integral boundary conditions are encountered in various applications
such as population dynamics, blood flow models, chemical engineering, cellular systems, underground water flow, heat transmission, plasma physics, thermoelasticity, etc.
Nonlocal conditions come up when values of the function on the boundary is connected to values inside the domain.

Nonlocal conditions are found to be more plausible than the standard initial conditions for the
formulation of some physical phenomena in certain problems of thermodynamics, elasticity and wave
propagation. Further details can be found in the work by Byszewski \cite{Byszewski1, Byszewski2}.

In recent years, boundary value problems of fractional differential equations with Riemann-Liouville fractional integral and Caputo fractional derivative in boundary conditions have achieved great deal of interest and attention of several researchers. Many authors have studied the existence of solution of the fractional boundary value problems under various boundary conditions and by different approaches. We refer the readers to the papers \cite{Agarwal, Bashir1, Bashir2, Bashir3, Khan, Lakoud1, Lakoud2, Lakoud3, Li, Su, Sud, Tarib1, Tarib2, Yang}.

Very recently, Agarwal et al. \cite{Agarwal} studied the following fractional order boundary value problem

\begin{equation*}
^{c}D^{q}x(t)=f(t,x(t)),\ 1<q\leq 2, \ t\in[0,1],
\end{equation*}

supplemented by boundary conditions, of the form

\begin{equation*}
x(0)=\delta x(\sigma), \ a\ ^{c}D^{p}x(\zeta_{1})+b\ ^{c}D^{p}x(\zeta_{2})=\sum_{i=1}^{m-2}\alpha_{i}x(\beta_{i}), \ 0<p<1.
\end{equation*}

Together with the above fractional differential equation they also investigated the boundary conditions

\begin{equation*}
x(0)=\delta_{1}\int_{0}^{\sigma}x(s)ds, \ {a}\ ^{c}D^{p}x(\zeta_{1})+{b}\  ^{c}D^{p}x(\zeta_{2})=\sum_{i=1}^{m-2}\alpha_{i}x(\beta_{i}), \ 0<p<1,
\end{equation*}
where $^{c}D^{q}$,$^{c}D^{p}$ denote the Caputo fractional derivatives of orders $q, p$ and $f:[0,1]\times\mathbb{R}\longrightarrow\mathbb{R}$ is a given continuous
function and $\delta, \delta_{1}, a, b, \alpha_{i}\in \mathbb{R}$, with $0<\sigma<\zeta_{1}<\beta_{1}<\beta_{2}<...<\beta_{m-2}<\zeta_{2}<1.$ \\
The existence and uniqueness results were proved via some well known tools of the fixed point theory.

In \cite{Bashir3}, the authors studied the existence and uniqueness of solutions to the fractional differential equation with four-point nonlocal Riemann-Liouville fractional integral boundary conditions of different order given by

\[^{c}D^{q}x(t)=f(t,x(t)),\ 1<q\leq 2, \ t\in[0,1],\]
\\
\[x(0)=a\int_{0}^{\eta}\frac{(\eta-s)^{\beta-1}}{\Gamma(\beta)}x(s)ds,\ 0<\beta\leq 1,\]
\\
\[x(1)=b\int_{0}^{\sigma}\frac{(\sigma-s)^{\alpha-1}}{\Gamma(\alpha)}x(s)ds,\ 0<\alpha\leq 1,\]

where $^{c}D^{q}$ is the Caputo fractional derivative of order $q$, $f$ is a given continuous
function, and $a, b, \eta, \sigma$ are real constants with $0<\eta, \sigma <1$, by using some
fixed point theorems.

Bashir et al. in \cite{Bashir2} discusses the existence and uniqueness of solutions of a new class of  fractional boundary value problems

\[^{c}D^{q}x(t)=f(t,x(t)), \ t\in[0,1], \ q\in(1,2],\]\\
\[x(0)=0, \ x(\xi)=a\int_{\eta}^{1}x(s)ds,\]
where $^{c}D^{q}$ denotes the Caputo fractional derivative of order $q$, $f$ is a given continuous function, and $a$ is a positive real constant, $\xi \in (0,1)$ with $\xi<\eta<1$.
The existence results are obtained with the aid of some classical fixed point theorems.

In \cite{Sud}, Sudsutad and Tariboon studied the existence and uniqueness of solutions for a boundary value problem of fractional order differential equation with three-point fractional integral boundary conditions given by

\[^{c}D^{q}x(t)=f(t,x(t)), \ t\in[0,1], \ q\in(1,2],\]\\
\[x(0)=0, \ x(1)={\alpha}\int_{0}^{\eta}\frac{(\eta-s)^{p-1}}{\Gamma(p)}x(s)ds,\ 0<\eta< 1,\ p>0, \]

where $^{c}D^{q}$ denotes the Caputo fractional derivative of order $q$, $f:[0,1]\times\mathbb{R}\longrightarrow\mathbb{R}$ is a continuous function and $\alpha \in \mathbb{R}$ is such that $\alpha\neq\ \Gamma(p+2)/\eta^{p+1}$.

Tariboon et al. \cite{Tarib2} have also studied the following fractional boundary value problem with three-point nonlocal Riemann-Liouville integral boundary conditions

\[D^{\alpha}x(t)=f(t,x(t)), \ 0<t<T, \ \alpha \in(1,2],\]\\
\[x(\eta)=0, \ I^{\nu}x(T)=\int_{0}^{T}\frac{(T-s)^{\nu-1}}{\Gamma(\nu)}x(s)ds=0,\]
where $D^{\alpha}$ denotes the Riemann-Liouville fractional derivative of order $\alpha>0$, $\eta \in(0, T)$ is a given constant.
The existence and uniqueness results were proved via the Banach contraction principle, the Banach's fixed point theorem and H\"{o}lder's inequality, the Krasnoselskii fixed point theorem and the Leray-Schauder nonlinear alternative.

In this paper, we introduce a new class of boundary value problems of fractional differential equations supplemented  with nonlocal Riemann-Liouville fractional integral and Caputo fractional derivative boundary conditions. In precise terms, we consider the following nonlocal problems:

\begin{equation} \label{eq-1.1}
^{c}D^{q}x(t)=f(t,x(t)),\ t\in[0,1],
\end{equation}

which includes one Caputo type fractional derivative, supplemented by boundary conditions
consisting of one fractional derivative of Caputo type and one Riemann-Liouville fractional integral

\begin{equation} \label{eq-1.2}
x^{\prime}(\xi)=\beta \ ^{c}D^{\nu}x(\eta), \ x(1)=\alpha [I^{p}x](\eta),
\end{equation}

where $^{c}D^{\mu}$ is the Caputo fractional derivative of order $\mu\in\{q,\nu\}$ such that $1<q\leq2$, $0<\nu\leq1$,  $I^{p}$ is the Riemann-Liouville fractional integral of order $p>0$, and $f:[0,1]\times\mathbb{R}\longrightarrow\mathbb{R}$ is a given continuous function, $0\leq\xi<\eta<1$ and $\alpha$, $\beta$ are appropriate real constants.

The  boundary conditions in \eqref{eq-1.2} implies that the value of the derivative of the unknown function at the nonlocal position $\xi$ is proportional to the value of the fractional derivative of the unknown function at the nonlocal position $\eta$, while the value of the unknown function at the right end point $(t=1)$ of the interval $[0,1]$ is proportional to the value of the  fractional integral of the unknown function at the nonlocal position $\eta$.

Motivated by the above recent works, the aim of this paper is to investigate the existence and uniqueness of solutions for the problem \eqref{eq-1.1}-\eqref{eq-1.2}.

Our analysis relies on some known fixed point theorems.

\section{Preliminaries\label{sec:2}}

In this section, we recall some basic definitions of fractional calculus and an auxiliary lemma to define the solution for the problem \eqref{eq-1.1}-\eqref{eq-1.2} is presented.

\begin{definition}
The Riemann-Liouville fractional integral of order $q$ for a continuous function $f$ is defined as
\begin{equation*}
I^{q}f(t)=\frac{1}{\Gamma(q)}\int_{0}^{t}\frac{f(s)}{(t-s)^{1-q}}ds, \ q>0,
\end{equation*}
provided the integral exists, where $\Gamma(.)$ is the gamma function, which is defined by $\Gamma(x)=\int_{0}^{\infty}t^{x-1}e^{-t}dt$.
\end{definition}

\begin{definition}
For at least n-times continuously differentiable function $f:[0,\infty)\rightarrow \mathbb{R}$, the Caputo derivative of fractional order $q$ is defined as
\begin{equation*}
^{c}D^{q}f(t)=\frac{1}{\Gamma(n-q)}\int_{0}^{t}\frac{f^{(n)}(s)}{(t-s)^{q+1-n}}ds,\ n-1<q<n,\ n=[q]+1,
\end{equation*}
where $[q]$ denotes the integer part of the real number $q$.
\end{definition}

\begin{lemma}[\cite{Kilbas}] \label{lem 2.1}
For $q > 0$, the general solution of the fractional differential equation $^{c}D^{q}x(t)=0$ is
given by
\begin{equation*}
x(t)=c_{0}+c_{1}t+...+c_{n-1}t^{n-1},
\end{equation*}
where $c_{i}\in \mathbb{R}$, $i = 0, 1,...,n-1 \ (n = [q] + 1)$.
\end{lemma}
According to Lemma \ref{lem 2.1}, it follows that
\begin{equation*}
{I^{q}}\ {^{c}D^{q}}x(t)=x(t)+c_{0}+c_{1}t+...+c_{n-1}t^{n-1},
\end{equation*}
for some $c_{i}\in \mathbb{R}$, $i = 0, 1,...,n-1 \ (n = [q] + 1)$.

\begin{lemma}[\cite{Podlubny}, \cite{Kilbas}]\label{lem 2.2}
If $\beta> \alpha>0$ and $x \in L_{1}[0,1]$, then
\item[(i)] ${^{c}D^{\alpha}}\ {I^{\beta}}x(t)={I^{\beta-\alpha}}x(t)$, holds almost everywhere on $[0,1]$ and it is valid at any point $t\in [0,1]$ if $x\in C[0,1]$;\  ${^{c}D^{\alpha}}\ {I^{\alpha}}x(t)=x(t) $, for all \ $t\in[0,1]$.
\item[(ii)] ${^{c}D^{\alpha}}t^{\lambda-1}=\frac{\Gamma(\lambda)}{\Gamma(\lambda-\alpha)}t^{\lambda-\alpha-1}$,
$\lambda>[\alpha]$ \ and \ ${^{c}D^{\alpha}}t^{\lambda-1}=0$, \  $\lambda<[\alpha]$.
\end{lemma}

\begin{lemma}\label{lem 2.3}
Let $\alpha\neq\frac{\Gamma(p+1)}{\eta^{p}}$ and  $\beta\neq\frac{\Gamma(2-\nu)}{\eta^{1-\nu}}$. Then, for any $h\in C([0, 1], \mathbb{R})$, the linear fractional boundary value problem
\begin{equation} \label{eq-2.1}
^{c}D^{q}x(t)=h(t),\ 0<t<1, \ 1<q\leq2,
\end{equation}
\begin{equation} \label{eq-2.2}
x^{\prime}(\xi)=\beta \ ^{c}D^{\nu}x(\eta), \ x(1)=\alpha [I^{p}x](\eta),\ 0<\nu\leq1,
\end{equation}
has an integral solution given by
\begin{eqnarray} \label{eq-2.3}
x(t)&=&\int_{0}^{t}\frac{(t-s)^{q-1}}{\Gamma(q)}h(s)ds-\frac{\Gamma(p+1)}{\Delta_{1}}\int_{0}^{1}\frac{(1-s)^{q-1}}{\Gamma(q)}h(s)ds \nonumber\\
&&+\frac{\Gamma(2-\nu)\Big((p+1)\Delta_{1}t+\Delta_{2}\Big)}{(p+1)\Delta_{1}\Delta_{3}}\int_{0}^{\xi}\frac{(\xi-s)^{q-2}}{\Gamma(q-1)}h(s)ds \nonumber\\
&&-\frac{\beta\Gamma(2-\nu)\Big((p+1)\Delta_{1}t+\Delta_{2}\Big)}{(p+1)\Delta_{1}\Delta_{3}}\int_{0}^{\eta}\frac{(\eta-s)^{q-\nu-1}}{\Gamma(q-\nu)}h(s)ds \nonumber\\
&&+\frac{\alpha p}{\Delta_{1}}\int_{0}^{\eta}\int_{0}^{s}\frac{(\eta-s)^{p-1}(s-\tau)^{q-1}}{\Gamma(q)}h(\tau)d\tau ds,
\end{eqnarray}
where
\begin{align} \label{eq-2.4}
\Delta_{1}&=\Gamma(p+1)-\alpha \eta^{p} \nonumber\\
\Delta_{2}&=\alpha \eta^{p+1}-\Gamma(p+2)\\
\Delta_{3}&=\beta \eta^{1-\nu}-\Gamma(2-\nu). \nonumber
\end{align}

\end{lemma}
\begin{proof}
From Lemma \ref{lem 2.1}, we may reduce \eqref{eq-2.1} to an equivalent integral equation
\begin{equation} \label{eq-2.5}
x(t)=I^{q}h(t)-c_{0}-c_{1}t,
\end{equation}
where $c_{0}, \ c_{1}\in \mathbb{R}$ are arbitrary constants. Consequently, the general solution of equation \eqref{eq-2.1} is
\begin{equation} \label{eq-2.6}
x(t)=\frac{1}{\Gamma(q)}\int_{0}^{t}(t-s)^{q-1}h(s)ds-c_{0}-c_{1}t,
\end{equation}
and
\begin{equation} \label{eq-2.7}
{x'}(t)=I^{q-1}h(t)-c_{1}.
\end{equation}
Now, in view of Lemma \ref{lem 2.2}, by taking the Caputo fractional derivative of order $\nu$ to both sides of \eqref{eq-2.6}, we get
\begin{equation} \label{eq-2.8}
^{c}D^{\nu}x(t)=I^{q-\nu}h(t)-c_{1}\frac{t^{1-\nu}}{\Gamma(2-\nu)}.
\end{equation}
From \eqref{eq-2.7} and \eqref{eq-2.8}, the boundary condition $x^{\prime}(\xi)=\beta \ ^{c}D^{\nu}x(\eta)$ implies that
\begin{equation*}
\frac{1}{\Gamma(q-1)}\int_{0}^{\xi}(\xi-s)^{q-2}h(s)ds-c_{1}=\beta\Big([I^{q-\nu}h](\eta)-c_{1}\frac{\eta^{1-\nu}}{\Gamma(2-\nu)}\Big),
\end{equation*}
which, on solving, yields

\begin{equation*}
c_{1}=\frac{\Gamma(2-\nu)}{\Delta_{3}}\Bigg(\beta \int_{0}^{\eta}\frac{(\eta-s)^{q-\nu-1}}{\Gamma(q-\nu)}h(s)ds
-\int_{0}^{\xi}\frac{(\xi-s)^{q-2}}{\Gamma(q-1)}h(s)ds\Bigg).
\end{equation*}

Using the Riemann-Liouville fractional integral of order $p$ for \eqref{eq-2.6}, we obtain
\begin{gather*}
\begin{aligned}
I^{p}x(t)&=\frac{1}{\Gamma(p)}\int_{0}^{t}(t-s)^{p-1}\Big(\frac{1}{\Gamma(q)}\int_{0}^{s}(s-\tau)^{q-1}h(\tau)d\tau-c_{0}-c_{1}s\Big)ds\\
&=\frac{1}{\Gamma(p)\Gamma(q)}\int_{0}^{t}\int_{0}^{s}(t-s)^{p-1}(s-\tau)^{q-1}h(\tau)d\tau ds-c_{0}\frac{t^{p}}{\Gamma(p+1)}-c_{1}\frac{t^{p+1}}{\Gamma(p+2)}.
\end{aligned}
\end{gather*}
The second condition of \eqref{eq-1.1} implies that

\begin{multline} \label{eq-2.9}
\frac{\alpha}{\Gamma(p)\Gamma(q)}\int_{0}^{\eta}\int_{0}^{s}(\eta-s)^{p-1}(s-\tau)^{q-1}h(\tau)d\tau ds-c_{0}\frac{\alpha \eta^{p}}{\Gamma(p+1)}-c_{1}\frac{\alpha  \eta^{p+1}}{\Gamma(p+2)}
\\
= \frac{1}{\Gamma(q)}\int_{0}^{1}(1-s)^{q-1}h(s)ds-c_{0}-c_{1},
\end{multline}

which, on inserting the value of $c_{1}$ in \eqref{eq-2.9}, we obtain

\begin{eqnarray}
c_{0} & = &\frac{1}{\Gamma(q)\Delta_{1}}\bigg\{\Gamma(p+1)\int_{0}^{1}(1-s)^{q-1}h(s)ds \nonumber\\
&&- \alpha p \int_{0}^{\eta}\int_{0}^{s}(\eta-s)^{p-1}(s-\tau)^{q-1}h(\tau)d\tau ds\bigg\} \nonumber\\
&&+ \frac{\Gamma(2-\nu)\Delta_{2}}{(p+1)\Delta_{1}\Delta_{3}}\bigg\{ \beta\int_{0}^{\eta}\frac{(\eta-s)^{q-\nu-1}}{\Gamma(q-\nu)}h(s)ds
\nonumber\\
&&- \int_{0}^{\xi}\frac{(\xi-s)^{q-2}}{\Gamma(q-1)}h(s)ds\bigg\}.
\end{eqnarray}

Substituting the values of $c_{0}$ and $c_{1}$ in \eqref{eq-2.6} we obtain the solution \eqref{eq-2.3}. This completes the proof.
\end{proof}

\section{Existences results\label{sec3}}
In this section, we establish sufficient conditions for the existence of solutions to the fractional order boundary value problem (1.1)-(1.2) using certain fixed point theorems.

Let $\mathcal{C}$ be the Banach space of all continuous functions from $[0,1]$ into $\mathbb{R}$ equipped
with the norm: $\|x\|=\sup\{|x(t)|, t\in[0,1 ]\}$. We define the operator $\mathfrak{S} : \mathcal{C}\rightarrow\mathcal{C}$ by
\begin{eqnarray} \label{eq-3.1}
(\mathfrak{S}x)(t)&=&\int_{0}^{t}\frac{(t-s)^{q-1}}{\Gamma(q)}f(s,x(s))ds-\frac{\Gamma(p+1)}{\Delta_{1}}\int_{0}^{1}\frac{(1-s)^{q-1}}{\Gamma(q)}f(s,x(s))ds \nonumber\\
&&+\frac{\Gamma(2-\nu)\Big((p+1)\Delta_{1}t+\Delta_{2}\Big)}{(p+1)\Delta_{1}\Delta_{3}}\int_{0}^{\xi}\frac{(\xi-s)^{q-2}}{\Gamma(q-1)}f(s,x(s))ds \nonumber\\
&&-\frac{\beta\Gamma(2-\nu)\Big((p+1)\Delta_{1}t+\Delta_{2}\Big)}{(p+1)\Delta_{1}\Delta_{3}}\int_{0}^{\eta}\frac{(\eta-s)^{q-\nu-1}}{\Gamma(q-\nu)}f(s,x(s))ds \nonumber\\
&&+\frac{\alpha p}{\Delta_{1}}\int_{0}^{\eta}\int_{0}^{s}\frac{(\eta-s)^{p-1}(s-\tau)^{q-1}}{\Gamma(q)}f(\tau,x(\tau))d\tau ds.
\end{eqnarray}
Obviously, the fixed points of the operator $\mathfrak{S}$ are the solutions of the fractional order boundary value
problem (1.1)-(1.2).

In order to prove our main results, the following well known fixed point theorems are needed.

\begin{theorem} [\cite{Gran}] \label{thm 3.1}
Let $(X,d)$ be a complete metric space and $T:X\rightarrow X$ be a contraction. Then $T$ has a unique fixed point in $X$.
\end{theorem}

\begin{theorem} [\cite{Smart}] \label{thm 3.2}
Let $X$ be a Banach space. Assume that $F:X\rightarrow X$ is a completely continuous operator and the set
$V=\{x\in X: x=\lambda{Fx},\ 0<\lambda<1$\} is bounded. Then $F$ has a fixed point in $X$.
\end{theorem}

\begin{theorem} [\cite{Smart}] \label{thm 3.3}
Let $E$ be a closed convex, bounded and nonempty subset of a Banach space $X$. Let $A,B$ be the operators such that
\begin{itemize}
\item [(1)] $Ax+By\in E$, for any $x, y\in E$;
\item [(2)]  $A$ is a completely continuous operator;
\item [(3)] $B$ is a contraction operator.
\end{itemize}
Then there exists at least one fixed point $z\in E$ such that $z = Az + Bz$.
\end{theorem}

In the following, for computational convenience, we set
\begin{eqnarray} \label{eq-3.2}
\Omega&=&\frac{1}{\Gamma(q+1)}+\frac{\Gamma(2-\nu)\big((p+1)|\Delta_{1}|+|\Delta_{2}|\big)}{(p+1)|\Delta_{1}\Delta_{3}|}\bigg
(\frac{\beta \eta^{q-\nu}}{\Gamma(q-\nu+1)}+\frac{\xi^{q-1}}{\Gamma(q)}\bigg) \nonumber\\
&&+\frac{\Gamma(p+1)}{|\Delta_{1}|}\bigg(\frac{1}{\Gamma(q+1)}+\frac{\alpha \eta^{p+q}}{\Gamma(p+q+1)}\bigg).
\end{eqnarray}
Now, we are in a position to present the main results of this paper. The first one existence result is based on Banach's contraction mapping principle \ref{thm 3.1}.

\begin{theorem} \label{thm 3.4}
Let $f:[0,1]\times\mathbb{R}\longrightarrow\mathbb{R}$ be a continuous function satisfying the Lipschitz
condition:\\
$(H_{1})$\ $|f(t,x)-f(t,y)|\leq L|x-y|$, \ $L>0$, \ $\forall t\in[0, 1]$,
\ $x,y\in \mathbb{R}$.\\
Then the problem \eqref{eq-1.1}-\eqref{eq-1.2} has a unique solution on $[0,1]$ provided that $L\Omega<1$, where $\Omega$ is given by \eqref{eq-3.2}.
\end{theorem}
\begin{proof}
Setting $\sup\{|f(t,0)|, t\in[0,1]\}=M<\infty$ and define $B_{\rho}=\{x\in \mathcal{C}:\|x\|\leq \rho\}$, where
\[\rho\geq \Omega{M}(1-\Omega{L})^{-1}.\]
As a first step, we show that $\mathfrak{S}B_{\rho}\subset B_{\rho}$. From $(H_{1})$, for $x\in B_{\rho}$, and $t\in[0,1]$, we get
\begin{gather} \label{eq-3.3}
\begin{aligned}
|f(t,x(t))|& \leq |f(t,x(t))-f(t,0)|+|f(t,0)| \\
&\leq L \|x\|+M\\
&\leq L {\rho}+M.
\end{aligned}
\end{gather}
Using \eqref{eq-3.1} and \eqref{eq-3.3}, we obtain
\begin{eqnarray*}
\|\mathfrak{S}x\|&\leq&\sup_{t\in[0,1]}\bigg\{\int_{0}^{t}\frac{(t-s)^{q-1}}{\Gamma(q)}|f(s,x(s))|ds+\frac{\Gamma(p+1)}{|\Delta_{1}|}\int_{0}^{1}\frac{(1-s)^{q-1}}{\Gamma(q)}|f(s,x(s))|ds\\
&&+\frac{\Gamma(2-\nu)\Big((p+1)|\Delta_{1}|+|\Delta_{2}|\Big)}{(p+1)|\Delta_{1}\Delta_{3}|}\int_{0}^{\xi}\frac{(\xi-s)^{q-2}}{\Gamma(q-1)}|f(s,x(s))|ds \end{eqnarray*}
\begin{eqnarray*}
&&+\frac{\beta\Gamma(2-\nu)\Big((p+1)|\Delta_{1}|+|\Delta_{2}|\Big)}{(p+1)|\Delta_{1}\Delta_{3}|}\int_{0}^{\eta}\frac{(\eta-s)^{q-\nu-1}}{\Gamma(q-\nu)}|f(s,x(s))|ds \\
&&+\frac{\alpha p}{|\Delta_{1}|}\int_{0}^{\eta}\int_{0}^{s}\frac{(\eta-s)^{p-1}(s-\tau)^{q-1}}{\Gamma(q)}|f(\tau,x(\tau))|d\tau ds\bigg\}\\
&\leq& (L {\rho}+M)\sup_{t\in[0,1]}\bigg\{\frac{t^{q}}{\Gamma(q+1)}+
\frac{\Gamma(2-\nu)\Big((p+1)|\Delta_{1}|+|\Delta_{2}|\Big)}{(p+1)\Gamma(q)|\Delta_{1}\Delta_{3}|}\xi^{q-1}\\
&&+\frac{\beta\Gamma(2-\nu)\Big((p+1)|\Delta_{1}|+|\Delta_{2}|\Big)}{(p+1)\Gamma(q-\nu+1)|\Delta_{1}\Delta_{3}|}\eta^{q-\nu}+\frac{\Gamma(p+1)}{\Gamma(q+1)|\Delta_{1}|}\\
&&+\frac{\alpha \Gamma(p+1)}{\Gamma(p+q+1)|\Delta_{1}|}\eta^{p+q} \bigg\}\\
&\leq& (L {\rho}+M) \Omega\leq\rho.
\end{eqnarray*}
Thus, $\mathfrak{S}B_{\rho}\subset B_{\rho}$. Now, for  $x,y\in \mathcal{C}$, we have

\begin{eqnarray*}
\|(\mathfrak{S}x)-(\mathfrak{S}y)\|&\leq&\sup_{t\in[0,1]}\bigg\{\int_{0}^{t}\frac{(t-s)^{q-1}}{\Gamma(q)}ds\\
&&+\frac{\Gamma(p+1)}{|\Delta_{1}|}\int_{0}^{1}\frac{(1-s)^{q-1}}{\Gamma(q)}ds\\
&&+\frac{\Gamma(2-\nu)\Big((p+1)|\Delta_{1}|t+|\Delta_{2}|\Big)}{(p+1)|\Delta_{1}\Delta_{3}|}\int_{0}^{\xi}\frac{(\xi-s)^{q-2}}{\Gamma(q-1)}ds \\
&&+\frac{\beta\Gamma(2-\nu)\Big((p+1)|\Delta_{1}|t+|\Delta_{2}|\Big)}{(p+1)|\Delta_{1}\Delta_{3}|}\int_{0}^{\eta}\frac{(\eta-s)^{q-\nu-1}}{\Gamma(q-\nu)}ds \\
&&+\frac{\alpha p}{|\Delta_{1}|}\int_{0}^{\eta}\int_{0}^{s}\frac{(\eta-s)^{p-1}(s-\tau)^{q-1}}{\Gamma(q)}d\tau ds\bigg\}L\|x-y\|\\
&\leq& L \Omega \|x-y\|.
\end{eqnarray*}
According to the condition $L\Omega<1$, it follows that the operator $\mathfrak{S}$ is a contraction. Therefore, by Theorem \ref{thm 3.1} (Banach's contraction principle) , there exists a unique fixed point in $B_{\rho}$ for the operator $\mathfrak{S}$  which is a unique solution for the problem \eqref{eq-1.1}-\eqref{eq-1.2}. This completes the proof.
\end{proof}

Now, we establish another existence result for BVP \eqref{eq-1.1}-\eqref{eq-1.2} by applying Schaefer's fixed point Theorem \ref{thm 3.2}.

\begin{theorem} \label{thm 3.5}
Let $f:[0,1]\times\mathbb{R}\longrightarrow\mathbb{R}$ be a continuous function satisfying the assumption $(H_{1})$. In addition, it is assumed that\\
$(H_{2})$ there exists a constant  $L>0$ such that $|f(t,x)|\leq L$,\ for all $t\in[0, 1]$, $x\in \mathbb{R} $.\\
Then the problem \eqref{eq-1.1}-\eqref{eq-1.2} has at least one solution on $[0,1]$.
\end{theorem}

\begin{proof}
We prove that the operator $\mathfrak{S}$ defined by \eqref{eq-3.1} has a fixed point by utilizing Schaefer's fixed
point theorem. The proof consists of several steps. Firstly, we show that the operator $\mathfrak{S}$ is continuous.

Let $x_{n}$ be a sequence such that $x_{n}\rightarrow x$ in $\mathcal{C}$. Then for each $t\in[0,1]$, we have
\begin{eqnarray*}
|(\mathfrak{S}x_{n})(t)-(\mathfrak{S}x)(t)|&\leq&\bigg\{\int_{0}^{t}\frac{(t-s)^{q-1}}{\Gamma(q)}ds+\frac{\Gamma(p+1)}{|\Delta_{1}|}\int_{0}^{1}\frac{(1-s)^{q-1}}{\Gamma(q)}ds\\
&&+\frac{\Gamma(2-\nu)\Big((p+1)|\Delta_{1}|t+|\Delta_{2}|\Big)}{(p+1)|\Delta_{1}\Delta_{3}|}\int_{0}^{\xi}\frac{(\xi-s)^{q-2}}{\Gamma(q-1)}ds \\
&&+\frac{\beta\Gamma(2-\nu)\Big((p+1)|\Delta_{1}|t+|\Delta_{2}|\Big)}{(p+1)|\Delta_{1}\Delta_{3}|}\int_{0}^{\eta}\frac{(\eta-s)^{q-\nu-1}}{\Gamma(q-\nu)}ds \\
&&+\frac{\alpha p}{|\Delta_{1}|}\int_{0}^{\eta}\int_{0}^{s}\frac{(\eta-s)^{p-1}(s-\tau)^{q-1}}{\Gamma(q)}d\tau ds\bigg\}\\
&&\times\|f(. , x_{n}(.))-f(. , x(.))\|\\
&\leq& \Omega \|f(. , x_{n}(.))-f(. , x(.))\|.
\end{eqnarray*}
Since $f$ is continuous , then $\|\mathfrak{S}x_{n}-\mathfrak{S}x\|\rightarrow 0$ as $n \rightarrow \infty$. Therefore $\mathfrak{S}$ is continuous.

Now, it will be shown that $\mathfrak{S}$  maps bounded sets into bounded sets in $\mathcal{C}$.\
For $\rho>0$, let $B_{\rho}=\{x\in \mathcal{C}:\|x\|\leq \rho\}$ be bounded set in $\mathcal{C}$. \
In view of the condition $(H_{2})$, it is easy to establish that $\|\mathfrak{S}x\|\leq L \Omega=M$, $x\in B_{\rho}$.

Thus $\mathfrak{S}$ is uniformly bounded on $B_{\rho}$. Moreover, for $t_{1}, t_{2}\in[0,1]$ with $t_{1}<t_{2}$ and $x\in B_{\rho}$, we get the following estimates
\begin{eqnarray*}
|(\mathfrak{S}x)(t_{2})-(\mathfrak{S}x)(t_{1})|&\leq&\int_{0}^{t_{1}}\frac{(t_{2}-s)^{q-1}-(t_{1}-s)^{q-1}}{\Gamma(q)}|f(s,x(s))|ds\\
&&+\int_{t_{1}}^{t_{2}}\frac{(t_{2}-s)^{q-1}}{\Gamma(q)}|f(s,x(s))|ds\\
&&+\frac{\beta\Gamma(2-\nu)(t_{2}-t_{1})}{\Gamma(q-\nu)|\Delta_{3}|}\int_{0}^{\eta}(\eta-s)^{q-\nu-1}|f(s,x(s))|ds \\
&&+\frac{\Gamma(2-\nu)(t_{2}-t_{1})}{\Gamma(q-1)|\Delta_{3}|}\int_{0}^{\xi}(\xi-s)^{q-2}|f(s,x(s))|ds \\
&\leq& L\bigg\{\frac{1}{\Gamma(q+1)}\left[(t_{2}^{q}-t_{1}^{q})+2(t_{2}-t_{1})^{q}\right]+\frac{\Gamma(2-\nu)}{|\Delta_{3}|}\\
&&\times\bigg[\frac{\beta \eta^{q-\nu}}{\Gamma(q-\nu+1)}+\frac{\xi^{q-1}}{\Gamma(q)}\bigg](t_{2}-t_{1})\bigg\}
\end{eqnarray*}
As $t_{2}\rightarrow t_{1}$, the right-hand side tends to zero independently of $x\in B_{\rho}$. Thus, by the Arzel\'{a}-Ascoli theorem, the operator $\mathfrak{S}$ is completely continuous.

Next, we need to show that the set $\mathcal{V}=\{x\in \mathcal{C}: x=\lambda{\mathfrak{S}}x,\ 0<\lambda<1\}$ is bounded.

Let $x\in \mathcal{V}$ and $t\in[0,1]$. Then
\begin{eqnarray*}
x(t)&=&\lambda\bigg\{\int_{0}^{t}\frac{(t-s)^{q-1}}{\Gamma(q)}f(s,x(s))ds-\frac{\Gamma(p+1)}{\Delta_{1}}\int_{0}^{1}\frac{(1-s)^{q-1}}{\Gamma(q)}f(s,x(s))ds \nonumber\\
&&+\frac{\Gamma(2-\nu)\Big((p+1)\Delta_{1}t+\Delta_{2}\Big)}{(p+1)\Delta_{1}\Delta_{3}}\int_{0}^{\xi}\frac{(\xi-s)^{q-2}}{\Gamma(q-1)}f(s,x(s))ds \nonumber\\
&&-\frac{\beta\Gamma(2-\nu)\Big((p+1)\Delta_{1}t+\Delta_{2}\Big)}{(p+1)\Delta_{1}\Delta_{3}}\int_{0}^{\eta}\frac{(\eta-s)^{q-\nu-1}}{\Gamma(q-\nu)}f(s,x(s))ds \nonumber\\
&&+\frac{\alpha p}{\Delta_{1}}\int_{0}^{\eta}\int_{0}^{s}\frac{(\eta-s)^{p-1}(s-\tau)^{q-1}}{\Gamma(q)}f(\tau,x(\tau))d\tau ds\bigg\},
\end{eqnarray*}
which implies using $\lambda<1$ that
\[\|x\|=\sup_{t\in[0,1]}\{|\lambda( \mathfrak{S}x)(t)|\}\leq L\Omega=M.\]
Therefore, $\mathcal{V}$ is bounded. By Schaefer's fixed point Theorem \ref{thm 3.2}, we conclude that the operator $\mathfrak{S}$ has a fixed point which is a solution of the fractional order boundary value problem \eqref{eq-1.1}-\eqref{eq-1.2}. This completes the proof.
\end{proof}

Our next result on existence and uniqueness is based on Krasnoselskii's fixed point Theorem \ref{thm 3.3}.

\begin{theorem}\label{thm 3.6}
Let $f:[0,1]\times\mathbb{R}\longrightarrow\mathbb{R}$ be a continuous function satisfying the assumption $(H_{1})$. Moreover, it is assumed that\\
$(H_{3})$ $|f(t,x)|\leq \sigma(t)$,\ $\forall (t,x)\in[0,1]\times \mathbb{R}$, where ${\sigma}\in C([0, 1], \mathbb{R^{+}})$.\\
Then the boundary value problem \eqref{eq-1.1}-\eqref{eq-1.2} has at least one solution on $[0,1]$ if
\[L \bigg(\Omega-\frac{1}{\Gamma(q+1)}\bigg)<1,\]
where $\Omega$ is given by \eqref{eq-3.2}.
\end{theorem}

\begin{proof}
If we denote $B_{\rho}=\{x\in \mathcal{C}:\|x\|\leq\rho\}$, where $\rho\geq\Omega\|\sigma\|$ with $\|\sigma\|=\sup_{t\in[0,1]}|\sigma(t)|$, and $\Omega$ is given by \eqref{eq-3.2}. Then $B_{\rho}$ is a bounded closed convex subset of $\mathcal{C}$.

For $t\in[0,1]$, we define two operators on $B_{\rho}$ as
\begin{eqnarray*}
(\mathfrak{S_{1}}x)(t)&=&\int_{0}^{t}\frac{(t-s)^{q-1}}{\Gamma(q)}f(s,x(s))ds\\
(\mathfrak{S_{2}}x)(t)&=&\frac{\Gamma(2-\nu)\Big((p+1)\Delta_{1}t+\Delta_{2}\Big)}{(p+1)\Delta_{1}\Delta_{3}}\int_{0}^{\xi}\frac{(\xi-s)^{q-2}}{\Gamma(q-1)}f(s,x(s))ds
 \nonumber\\
&&+\frac{\alpha p}{\Delta_{1}}\int_{0}^{\eta}\int_{0}^{s}\frac{(\eta-s)^{p-1}(s-\tau)^{q-1}}{\Gamma(q)}f(\tau,x(\tau))d\tau ds
\nonumber\\
&&-\frac{\Gamma(p+1)}{\Delta_{1}}\int_{0}^{1}\frac{(1-s)^{q-1}}{\Gamma(q)}f(s,x(s))ds\nonumber\\
&&-\frac{\beta\Gamma(2-\nu)\Big((p+1)\Delta_{1}t+\Delta_{2}\Big)}{(p+1)\Delta_{1}\Delta_{3}}\int_{0}^{\eta}\frac{(\eta-s)^{q-\nu-1}}{\Gamma(q-\nu)}f(s,x(s))ds.
\end{eqnarray*}
For $x,y\in B_{\rho}$, we find that $\|\mathfrak{S_{1}}x+\mathfrak{S_{2}}y\|\leq \Omega\|\sigma\|\leq\rho$, which implies that $\mathfrak{S_{1}}x+\mathfrak{S_{2}}y\in  B_{\rho}$.

Using $(H_{1})$ and \eqref{eq-3.2}, for  $x,y\in \mathcal{C}$, we obtain
\begin{eqnarray*}
\|(\mathfrak{S_{2}}x)-(\mathfrak{S_{2}}y)\|&\leq&\sup_{t\in[0,1]}\bigg\{\frac{\Gamma(2-\nu)\Big((p+1)|\Delta_{1}|t+|\Delta_{2}|\Big)}{(p+1)|\Delta_{1}\Delta_{3}|}\int_{0}^{\xi}\frac{(\xi-s)^{q-2}}{\Gamma(q-1)}ds \\
&&+\frac{\alpha p}{|\Delta_{1}|}\int_{0}^{\eta}\int_{0}^{s}\frac{(\eta-s)^{p-1}(s-\tau)^{q-1}}{\Gamma(q)}d\tau ds
\\
&&+\frac{\Gamma(p+1)}{|\Delta_{1}|}\int_{0}^{1}\frac{(1-s)^{q-1}}{\Gamma(q)}ds
 \\
&&+\frac{\beta\Gamma(2-\nu)\Big((p+1)|\Delta_{1}|t+|\Delta_{2}|\Big)}{(p+1)|\Delta_{1}\Delta_{3}|}\int_{0}^{\eta}\frac{(\eta-s)^{q-\nu-1}}{\Gamma(q-\nu)}ds  \bigg\}L\|x-y\|\\
&\leq& L \left(\Omega-1/\Gamma(q+1)\right) \|x-y\|,
\end{eqnarray*}
which shows that the operator $\mathfrak{S_{2}}$ is a contraction since $L (\Omega-1/\Gamma(q+1))<1$.

For $x\in B_{\rho}$, we have
\[\|\mathfrak{S_{1}}x\|\leq\sup_{t\in[0,1]}\bigg\{\int_{0}^{t}\frac{(t-s)^{q-1}}{\Gamma(q)}|f(s,x(s))|ds\bigg\}\leq \frac{\|\sigma\|}{\Gamma(q+1)}.\]
Therefore, $\mathfrak{S_{1}}$ is uniformly bounded on $ B_{\rho}$. Now, we prove the compactness of the operator $\mathfrak{S_{1}}$.

Let $t_{1}, t_{2} \in [0; 1]$ with $t_{1}<t_{2}$  and $x\in B_{\rho}$. Then, we obtain
\begin{eqnarray*}
|(\mathfrak{S_{1}}x)(t_{2})-(\mathfrak{S_{1}}x)(t_{1})|&\leq&\int_{0}^{t_{1}}\frac{(t_{2}-s)^{q-1}-(t_{1}-s)^{q-1}}{\Gamma(q)}|f(s,x(s))|ds\\
&&+\int_{t_{1}}^{t_{2}}\frac{(t_{2}-s)^{q-1}}{\Gamma(q)}|f(s,x(s))|ds\\
&\leq& \frac{\tilde{f}}{\Gamma(q+1)}\bigg((t_{2}^{q}-t_{1}^{q})+2(t_{2}-t_{1})^{q}\bigg),
\end{eqnarray*}
where $\sup_{(t,x)\in[0,1]\times {B_{\rho}}}|f(t,x)|=\tilde{f}$. Obviously, the right-hand side of the above inequality tends to zero independently of $x\in B_{\rho}$ as $t_{2}\rightarrow t_{1}$. So $\mathfrak{S_{1}}$ is relatively compact on
$B_{\rho}$ . Hence, by the Arzel\'{a}-Ascoli theorem, $\mathfrak{S_{1}}$ is compact on $B_{\rho}$. Continuity of $f$ implies that the operator $\mathfrak{S_{1}}$ is continuous. Therefore, $\mathfrak{S_{1}}$ is completely continuous.
Thus all the hypothesis of Theorem \ref{thm 3.3} are satisfied and consequently the
problem \eqref{eq-1.1}-\eqref{eq-1.2} has at least one solution on $[0,1]$. This completes the proof.
\end{proof}

\section{Examples\label{sec:4}}

\begin{exmp}
Consider the following fractional boundary value problem

\begin{equation} \label{eq-4.1}
\begin{cases}
^{c}D^{\frac{4}{3}}x(t)=\frac{e^{-\cos^{2}t}}{(35e^{t}+1)\sqrt{t+16}}\sin{x}+\frac{\sqrt{2}}{2}\frac{t}{t+1},\ t\in[0,1],
\\
x^{\prime}(\frac{1}{3})=\ ^{c}D^{\frac{1}{2}}x(\frac{1}{2}), \ x(1)=3 [I^{\frac{3}{2}}x](\frac{1}{2}).
\end{cases}
\end{equation}

Here, $\alpha=3$, $\beta=1$, $\eta=\frac{1}{2}$, $\xi=\frac{1}{3}$, $q=\frac{4}{3}$, $p=\frac{3}{2}$, $\nu=\frac{1}{2}$, and $f(t,x)=\frac{e^{-\cos^{2}t}}{(35e^{t}+1)\sqrt{t+16}}\sin{x}+\frac{\sqrt{2}}{2}\frac{t}{t+1}$. With the given values, it is easy to see that $\alpha=3\neq\frac{\Gamma(p+1)}{\eta^{p}}=3\sqrt{\frac{\pi}{2}}$, $\beta=1\neq\frac{\Gamma(2-\nu)}{\eta^{1-\nu}}=\sqrt{\frac{\pi}{2}}$, $\Delta_{1}=\frac{3}{2}|\Delta_{3}|$ with $|\Delta_{3}|=\frac{1}{2}(\sqrt{\pi}-\sqrt{2})$, $|\Delta_{2}|=\frac{3}{8}(5\sqrt{\pi}-\sqrt{2})$, and $\Omega\simeq 40.4684$. \\
Clearly, $L=\frac{1}{144}$ as $|f(t,x)-f(t,y)|\leq \frac{1}{144}|x-y|$. Furthermore, upon computation, we get
\begin{eqnarray*}
L\Omega&=& L\bigg\{\frac{1}{\Gamma(q+1)}+\frac{\Gamma(2-\nu)\big((p+1)|\Delta_{1}|+|\Delta_{2}|\big)}{(p+1)|\Delta_{1}\Delta_{3}|}\bigg
(\frac{\beta \eta^{q-\nu}}{\Gamma(q-\nu+1)}+\frac{\xi^{q-1}}{\Gamma(q)}\bigg) \nonumber\\
&&+\frac{\Gamma(p+1)}{|\Delta_{1}|}\bigg(\frac{1}{\Gamma(q+1)}+\frac{\alpha \eta^{p+q}}{\Gamma(p+q+1)}\bigg)\bigg\} \\
&\simeq& 0.2810<1 .
\end{eqnarray*}
Thus, for the given boundary value problem \eqref{eq-4.1}, all the conditions of Theorem \ref{thm 3.4} are satisfied. So, by Theorem \ref{thm 3.4}, there exists a unique solution for the problem \eqref{eq-4.1} on $[0,1]$.

\end{exmp}

\begin{exmp}
Consider a fractional boundary value problem given by

\begin{equation} \label{eq-4.2}
\begin{cases}
^{c}D^{\frac{3}{2}}x(t)=\frac{e^{-2t}[2+\sin(t^{2}-t)]}{(3+|\cos{x}|)(4+t^{3}x^{2})^{2}},\ t\in[0,1],
\\
x^{\prime}(\frac{1}{5})= \frac{1}{4} \ ^{c}D^{\frac{2}{3}}x(\frac{3}{4}), \ x(1)= \frac{1}{2}[I^{\frac{4}{3}}x](\frac{3}{4}),
\end{cases}
\end{equation}
where,  $\alpha=\frac{1}{2}$, $\beta=\frac{1}{4}$, $\eta=\frac{3}{4}$, $\xi=\frac{1}{5}$, $q=\frac{3}{2}$, $p=\frac{4}{3}$, $\nu=\frac{2}{3}$, and $f(t,x)=\frac{e^{-2t}[2+\sin(t^{2}-t)]}{(3+|\cos{x}|)(4+t^{3}x^{2})^{2}}$. By simple calculations, we find that
$ \frac{\Gamma(2-\nu)}{\eta^{1-\nu}}=  \frac{(1/3)\Gamma(1/3)}{(\frac{3}{4})^{\frac{1}{3}}}$, $   \frac{\Gamma(p+1)}{\eta^{p}}= (\frac{4}{3})^{2} \frac{\Gamma(2-\nu)}{\eta^{1-\nu}}$. Then,
\[\alpha=\frac{1}{2}\neq\frac{\Gamma(p+1)}{\eta^{p}}\simeq 1.7473, \ \text{and} \ \beta=\frac{1}{4}\neq\frac{\Gamma(2-\nu)}{\eta^{1-\nu}}\simeq 0.9829.\]

We easily get $|f(t,x)|\leq\frac{1}{16}$. Hence, all the conditions of Theorem \ref{thm 3.5} are satisfied. Thus, by Theorem \ref{thm 3.5} the fractional order boundary value problem \eqref{eq-4.2} has at least one solution on $[0, 1]$:
\end{exmp}

\begin{exmp}
As a third example we consider the fractional boundary value problem
\begin{equation} \label{eq-4.3}
\begin{cases}
^{c}D^{\frac{5}{4}}x(t)=\frac{e^{-t}\cos(t\sqrt{2})}{(1+|x|)(2+e^{t})^{2}\sqrt{t+25}}+\frac{t}{t+1},\ t\in[0,1],
\\
x^{\prime}(\frac{2}{5})= \frac{3}{2} \ ^{c}D^{\frac{1}{4}}x(\frac{5}{7}), \ x(1)= 2[I^{\frac{5}{3}}x](\frac{5}{7}),
\end{cases}
\end{equation}
where, $\alpha=2$, $\beta=\frac{3}{2}$, $\eta=\frac{5}{7}$, $\xi=\frac{2}{5}$, $q=\frac{5}{4}$, $p=\frac{5}{3}$, $\nu=\frac{1}{4}$, and $f(t,x)=\frac{e^{-t}\cos(t\sqrt{2})}{(1+|x|)(2+e^{t})^{2}\sqrt{t+25}}+\frac{t}{t+1}$. With the given values, it is found that $\alpha=2\neq\frac{\Gamma(p+1)}{\eta^{p}}=\frac{(10/9)\Gamma(2/3)}{(\frac{5}{7})^{\frac{5}{3}}}\simeq 2.6361$, $\beta=\frac{3}{2}\neq\frac{\Gamma(2-\nu)}{\eta^{1-\nu}}=\frac{(3/4)\Gamma(3/4)}{(\frac{5}{7})^{\frac{3}{4}}}\simeq 1.1829$, $\Delta_{1}=0.3631$, $|\Delta_{2}|=3.1968$, $\Delta_{3}=0.2464$, and $\big(\Omega-\frac{1}{\Gamma(q+1)}\big)\simeq 35.5044$. \\
Since $|f(t,x)-f(t,y)|\leq \frac{1}{45}|x-y|$, then $(H_{1})$ is satisfied with $L=\frac{1}{45}$. Further,
\begin{equation*}
|f(t,x)|=\bigg|\frac{e^{-t}\cos(t\sqrt{2})}{(1+|x|)(2+e^{t})^{2}\sqrt{t+25}}+\frac{t}{t+1}\bigg|\leq \frac{e^{-t}}{45}+\frac{t}{t+1}.
\end{equation*}
Obviously all the conditions of Theorem \ref{thm 3.6} are satisfied with $L(\Omega-1/\Gamma(q+1))\simeq 0.7890<1$.
Hence, by Theorem \ref{thm 3.6}, the fractional order boundary value problem \eqref{eq-4.3} has at least one solution on $[0, 1]$.
\end{exmp}





\begin{thebibliography}{}

\bibitem{Agarwal}
R. P. Agarwal, A. Alsaedi, A. Alsharif, B. Ahmad
, \textit{On nonlinear fractional-order boundary value problems with nonlocal multi-point conditions involving Liouville-Caputo derivatives}, Differ. Equ. Appl., Volume 9, Number 2 (2017), 147–-160, doi:10.7153/dea-09-12.

\bibitem{Bashir1}
B. Ahmad, S. K. Ntouyas, A. Alsaedi, \textit{Fractional differential equations and inclusions with nonlocal generalized Riemann-Liouville integral boundary conditions}, International Journal of Analysis and Applications., Volume 13, Number 2 (2017), 231--247.

\bibitem{Bashir2}
B. Ahmad, A. Alsaedi, A. Assolami, R. P. Agarwal, \textit{A new class of fractional boundary value problems},  Adv. Difference Equ., 2013, 2013: 373.

\bibitem{Bashir3}
B. Ahmad, S. K. Ntouyas, A. Assolami, \textit{Caputo type fractional differential equations with nonlocal Riemann-Liouville integral boundary conditions}, J. Appl. Math. Comput., (2013) 41:339-–350.

\bibitem{Bai}
C. Z. Bai, \textit{Triple positive solutions for a boundary value problem of nonlinear fractional differential equation}, Electron. J. Qual. Theory Diff. Equ.,  \textbf{24} (2008), 1--10.

\bibitem{Bakakhani}
A. Bakakhani, V. D. Gejji, \textit{Existence of positive solutions of nonlinear fractional differential equations}, J. Math. Anal. Appl., \textbf{278} (2003), 434--442.


\bibitem{Byszewski1}
L. Byszewski, V. Lakshmikantham, \textit{Theorem about the existence and uniqueness of a solution of a nonlocal abstract Cauchy problem in a Banach space}, Appl. Anal., \textbf{40} (1991), 11--19.

\bibitem{Byszewski2}
L. Byszewski, \textit{Theorems about existence and uniqueness of solutions of a semilinear evolution nonlocal Cauchy problem}, J. Math. Anal. Appl., \textbf{162} (1991), 494--505.

\bibitem{Gran}
A. Granas, J. Dugundji, \textit{Fixed point theory}, Springer-Verlag, New York, 2003.

\bibitem{Kaufmann}
E. R. Kaufmann, E. Mboumi, \textit{Positive solutions of a boundary value problem for a nonlinear fractional differential equation}, Electron. J. Qual. Theory Diff. Equ., \textbf{3} (2008), 1--11.

\bibitem{Khan}
R. A. Khan, H. Khan , \textit{Existence of solution for a three point boundary value problem of fractional differential equation}, J. Fract. Calc. Appl., Vol. 5(1) (2014), 156--164.

\bibitem{Kilbas}
A. A. Kilbas, H. M. Srivastava, J. J. Trujillo,\textit{Theory and Applications of Fractional
Differential Equations}, North-Holland Mathematics Studies, vol. 204. Elsevier, Amsterdam, 2006.

\bibitem{Kosmatov}
N. Kosmatov, \textit{A singular boundary value problem for nonlinear differential equations of fractional order}, J. Appl. Math. Comput., 29 (1-2) (2009), 125--135.

\bibitem{Lakoud1}
A. Guezane-Lakoud, R. Khaldi, \textit{Positive solution to a higher order fractional boundary value problem with fractional integral condition}, Rom. J. Math. Comput. Sci., 2 (2012), 41--54.

\bibitem{Lakoud2}
A. Guezane-Lakoud, R. Khaldi, \textit{Solvability of a three-point fractional nonlinear boundary value problem}, Differ. Equ. Dyn. Syst., 20 (2012), 395--403.

\bibitem{Lakoud3}
A. Guezane-Lakoud, R. Khaldi, \textit{Solvability of a fractional boundary value problem with fractional integral condition}, Nonlinear Anal., 75 (2012), 2692--2700.


\bibitem{Lakshmikantham}
V. Lakshmikantham, \textit{Theory of fractional functional differential equations}, Nonlinear Anal., \textbf{69} (2008), 3337--3343.

\bibitem{Li}
C. F. Li, X. N. Luo, Y. Zhou, \textit{Existence of positive solutions of the boundary value problem for nonlinear fractional differential equations}, Comput. Math. Appl., \textbf{59} (2010), 1363--1375.


\bibitem{Meral}
F. C. Meral, T. J. Royston, R. Magin, \textit{Fractional calculus in viscoelasticity: an experimental
study}, Commun. Nonlinear Sci. Numer. Simul.,  \textbf{15}(2010), 939--945.

\bibitem{Miller}
K. S. Miller, B. Ross, \textit{An Introduction to the Fractional Calculus and Fractional Differential Equations}, Wiley, New York, 1993.

\bibitem{Nigmatullin}
R. Nigmatullin, T. Omay, D. Baleanu, \textit{On fractional filtering versus conventional filtering in
economics}, Commun. Nonlinear Sci. Numer. Simul., \textbf{15} (2010), 979--986.

\bibitem{Oldham}
K. Oldham, \textit{Fractional differential equations in electrochemistry}, Adv. Eng. Softw., \textbf{41} (2010), 9--12.

\bibitem{Orsingher}
E. Orsingher, L. Beghin, \textit{Time-fractional telegraph equations and telegraph processes with
brownian time}, Probab. Theory. Related. Fields., \textbf{128} (2004), 141-160.

\bibitem{Podlubny}
I. Podlubny, \textit{Fractional Differential Equations}, Academic Press, Inc., San Diego, 1999.

\bibitem{Smart}
D. R. Smart, \textit{Fixed Point Theorems}, Cambridge University Press, London-New York, 1974.

\bibitem{Su}
C. M. Su, J. P. Sun, Y. H. Zhao, \textit{Existence and Uniqueness of Solutions for BVP of Nonlinear Fractional Differential Equation}, Int. J. Differ. Equ., Vol. 2017, Article ID 4683581, 7 pages. doi.org/10.1155/2017/4683581.


\bibitem{Sud}
W. Sudsutad, J. Tariboon, \textit{Boundary value problems for fractional differential equations with three-point fractional integral boundary conditions}, Adv. Difference Equ., 2012, 2012:93.

\bibitem{Tarib1}
J. Tariboon, S. K. Ntouyas, W. Sudsutad, \textit{Positive solutions for fractional differential equations with three-point multi-term fractional integral boundary conditions}, Adv. Difference Equ., 2014, 2014: 28.

\bibitem{Tarib2}
J. Tariboon, T. Sitthiwirattham, S. K. Ntouyas, \textit{Boundary value problems for a new class of three-point nonlocal Riemann-Liouville integral boundary conditions}, Adv. Difference Equ., 2013, 2013: 213.


\bibitem{Yang}
W. Yang, \textit{Positive solutions for nonlinear Caputo fractional differential equations with integral boundary conditions}, J. Appl. Math. Comput., \textbf{44} (1-2) (2014), 39--59. doi:10.1007/s12190-013-0679-8.


\bibitem{Zhou1}
Yong Zhou, \textit{Existence and uniqueness of solutions for a system of fractional differential equations}, J. Frac. Calc. Appl. Anal., \textbf{12} (2) (2009), 195--204.

\bibitem{Zhou2}
Yong Zhou, \textit{Existence and uniqueness of fractional functional differential equations with unbounded delay}, Int. J. Dyn. Syst. Differ. Equ.,  \textbf{1} (4)(2008), 239--244.

\end{thebibliography}
\end{document}